\documentclass[11pt,twoside]{amsart}
\usepackage{graphicx}
\usepackage{color}
\usepackage{amssymb}
\usepackage{amsmath}
\usepackage{amsthm}
\usepackage{amsfonts}
\usepackage{amsmath, amsfonts, amssymb}
\newtheorem{theorem}{Theorem}[section]
\newtheorem{corollary}[theorem]{Corollary}
\newtheorem{lemma}[theorem]{Lemma}
\newtheorem{example}[theorem]{Example}
\newtheorem{proposition}[theorem]{Proposition}

\newtheorem{remark}[theorem]{Remark}

\newtheorem{definition}[theorem]{Definition}
\numberwithin{equation}{section}

\begin{document}
	\makeatletter
	\makeatother
	\small \begin{center}{\textit{ In the name of
				Allah, the Beneficent, the Merciful.}}\end{center}\vspace{0.5cm}
	\large
	\begin{center}
		\textbf{AN OPEN SUBSET OF THE VARIETY OF $n$- DIMENSIONAL ALGEBRAS, CLASSIFICATION AND AUTOMORPHISMS}\\
		\textbf{Ural Bekbaev}\\
		\smallskip
		{Turin Polytechnic University in Tashkent, Tashkent, Uzbekistan;}\\
		{uralbekbaev@gmail.com}
	\end{center}
	\small \begin{center} Abstract\end{center}
	Classification, up to isomorphism, of algebras from a non-empty subset of the variety of  $n$- dimensional algebras is presented. It is shown that these algebras have only trivial automorphism and if the basic field is algebraically closed then it is an open dense subset of the variety of $n$- dimensional algebras. Properties of these sets, depending on $n$, with respect to the direct sum and tensor product are considered. Moreover one more way of construction of such subsets with similar properties is considered as well.\\
	Keyword: algebra, isomorphism, matrix of structure constants.\\
	MSC2010: 15A21, 15A63, 15A69, 17A45.
	\large
	
	\section{Introduction}  A classification, up to isomorphism, of algebraic structures is one of the core problems of algebra. In particular, classification of all algebra structures on $\mathbb{F}^n$, over an algebraically closed field $\mathbb{F}$ and natural $n$, is one of the such problems. It is believed that full solution of this problem for any $\mathbb{F}$ and $n$ is impossible. In the current paper we are going to show that if instead of all $n$-dimensional algebras one considers some dense subset of the variety of $n$-dimensional algebras then the classification problem can be solved positively. 
	
	In \cite{P} (Theorem 4) a theoretical existence of a non-empty open subset, consisting of simple algebras with trivial automorphism groups, of the variety of $n$-dimensional algebras is proved in $\mathbb{F}$ is algebraically closed, $Char(\mathbb{F})\neq 2$ case. The existence of normal(canonical) forms for algebras from that open set is shown as well(Theorem 2).   
	
	In \cite{BU1} we also have considered similar problem and offered an algorithm to construct a matrix $P(A)$, with a nice property, which helps classify, up to isomorphism, all algebras $A$ for which  $P(A)$ is invertible. Moreover  some results regarding invariants of algebras, under assumption of existence of such a $P(A)$, have been proved as well. Unfortunately, further explorations have shown that the algorithm offered does not work in that form for $n$, except $n=2$ case. 
	
	In $n=2$ case the classification problem is solved completely for the set of all algebras over any basic field $\mathbb{F}$ in \cite{BU2}. In \cite{ABR1, ABR2} it is shown how to get description of automorphism groups, derivation algebras, classification of different classes of two-dimensional algebras by the use of such classification result.
	
	In this paper we provide another algorithm for the construction of that matrix $P(A)$, show that it works well in arbitrary $n$ and $\mathbb{F}$ case, the set of algebras, for which $P(A)$ is invertible ($\mathcal{A}_{00}(n)$), is not empty if even $\mathbb{F}$ is not algebraically closed. We classify all algebras for which $P(A)$ is invertible, provide a formulae to find a canonical form of any such $A$, show that the automorphism group of $A$ is trivial and it does not have invertible derivation. Unlike \cite{P} we show that, if $\mathbb{A}$ is $n$-dimensional, $\mathbb{A'}$ is $n'$-dimensional algebras then $\mathbb{A}\oplus \mathbb{A'}\in\mathcal{A}_{00}(n+n')$ if and only if $\mathbb{A}\in\mathcal{A}_{00}(n)$ and $\mathbb{A'}\in\mathcal{A}_{00}(n')$. Moreover, it is shown that $\mathbb{A}\otimes \mathbb{A'}\in\mathcal{A}_{00}(nn')$ if and only if $\mathbb{A}\in\mathcal{A}_{00}(n)$ and $\mathbb{A'}\in\mathcal{A}_{00}(n')$. A description of elements of $\mathbb{A}\in\mathcal{A}_{00}(2)$, up to isomorphism, is given as well.
	
	It should be noted that there are many investigations dealing with classification of different types of $n$-dimensional PI-algebras, sometimes with some constraints, in the case of small values of $n$. There are investigations dealing with such problem for any value of $n$, for example, simple, semi-simple associative, Lie algebras and et cetera \cite{N,Bl}. But from topological point of view, in contrast to our case, they are small subsets of the variety of $n$-dimensional algebras, though those algebras are important in applications. Here 'small' subset means that its complement is a dense subset of the variety.  
	
	The paper is organised in the following way. In the next section some needed notions, notations and results are presented. Then the main result is described in another section. In the last section another way of construction of such subsets with similar properties is considered.

	\section{Preliminaries}
	Consider the following symmetric $n$-order square tridiagonal 
	matrix\\
	$C=\begin{pmatrix}	1&-1&0&0&...&0&0\\
		-1&2&-1&0&...&0&0\\
		0&-1&2&-1&...&0&0\\
		0&0&-1&2&...&0&0\\
		.&.&.&.&...&.&.\\
		.&.&.&.&...&2&-1\\
		0&0&0&0&...&-1&2\end{pmatrix}$ over a commutative, associative ring $R$ with the unit element $1$, its powers' diagonal elements $(C^k)_{ii}$ and construct square matrix $V=V(C)$ who's rows are $((C^k)_{11}, (C^k)_{22},...,(C^k)_{nn})$, where $k=0,1,2,...,n-1$ and it is assumed that $((C^0)_{11}, (C^0)_{22},...,(C^0)_{nn})=(1,1,1,...,1)$, that is $C^0$ is the identity matrix. 
	\begin{lemma} The equality $\det(V(C))=\pm 1$, depending on $n$, is true.	\end{lemma}
	\begin{proof} To prove one can use induction on $n$ in the following way. For small values of $n$ the lemma can be checked easily. The entries of $V$ have the following properties: 
		$$V_{2i,i+1}=...=V_{2i,n-(i-1)}, \ \mbox{whenever}\ 2i\leq n, V_{2i+1,i+1}=...=V_{2i+1,n-i}, \ \mbox{whenever}\ 2i\leq n-1,$$
		$$V_{n,k}-V_{n,k+1}=1, \ \mbox{if}\ n=2k-1, V_{n,k+1}-V_{n,k}=1, \ \mbox{if}\ n=2k.$$ Therefore the difference of $(k+1)_{th}$ and $k_{th}$ columns, in corresponding order, is $(0,0,...,0,1)$, which implies that the lemma is true.   
	\end{proof}
	
	Further the following notations are used.
	If $\mathbb{F}$ is any field, $A=(a_{ij})$, $B=(b_{kl})$ are rectangular matrices over it then $A\otimes B$ stands for the tensor product $A\otimes B=(a_{ij}B)$. Let $n$, $m$  be any natural numbers, $M(\mathbf{a}_1|\mathbf{a}_2|...|\mathbf{a}_n)=(\mathbf{a}_1/\mathbf{a}_2/.../\mathbf{a}_n)$ stand for the matrix consisting of the same size rows $\mathbf{a}_1,\mathbf{a}_2,...,\mathbf{a}_n\in \mathbb{F}^m$, where $(\mathbf{a}_1|\mathbf{a}_2|...|\mathbf{a}_n)$ stands for row matrix with 'components' $\mathbf{a}_1,\mathbf{a}_2,...,\mathbf{a}_n$. If $A_1,A_2,...,A_n$ are same size square matrices then $A=(A_1|A_2|...|A_n)$ stands for the corresponding row-block matrix, and $\overline{tr}(A)=(tr(A_1),tr(A_2),...,tr(A_n)),$ where $tr(A_i)$ means the ordinary trace of the matrix $A_i$, $I$ stands for $n$-order unit matrix, $I_m$ for $m$-order unit matrix.
	
	\begin{proposition}The equalities $$(\mathbf{b}_1|\mathbf{b}_2|...|\mathbf{b}_n)=(\mathbf{a}_1|\mathbf{a}_2|...|\mathbf{a}_n)(g\otimes h) \ \mbox{and}\ (\mathbf{b}_1/\mathbf{b}_2/.../\mathbf{b}_n)=g^t(\mathbf{a}_1/\mathbf{a}_2/.../\mathbf{a}_n)h$$ are equivalent, where 	$\mathbf{a}_i\in \mathbb{F}^m,\mathbf{b}_j\in \mathbb{F}^m$ are row vectors, $g\in Mat(n\times k,\mathbb{F})$, $h\in Mat(m\times m,\mathbb{F})$.
	\end{proposition}
	
	\begin{proof}Indeed, $\mathbf{b}_i=\sum_{k=1}^n \mathbf{a}_kg_{ki}h=\sum_{k=1}^n g_{ki}\mathbf{a}_kh=(g^tA)_ih$, where  $(g^tA)_i$ stands for the $i^{th}$ row of matrix $g^tA$.\end{proof}
	
	\begin{proposition} If $A'= gA(h\otimes g^{-1})$, where $A=(A_1|A_2|...|A_n)$, $g,A_i\in Mat(m,\mathbb{F})$,\\ $h\in Mat(n\times k,\mathbb{F})$ then
		$$\overline{tr}(A')=(tr(A'_1), tr(A'_2),...,tr(A'_k))=\overline{tr}(A)h,\ \mbox{more generally, for any natural\ k },$$ $$\overline{tr}((A')^{[k]})=\overline{tr}(A^{[k]})h^{\otimes k}$$ is true, where by definition $$A^{[1]}=A=(A_1|A_2|...|A_n),\ \ A^{ [k]}=(A_1A^{ [k-1]}|A_2A^{[k-1]}|...|A_nA^{[k-1]}).$$
	\end{proposition}
	
	\begin{proof} The equality $A'= gA(h\otimes g^{-1})$ implies that $A'=(A'_1|A'_2|...|A'_m)$ and $A'_i=\sum_{p=1}^n h_{pi}gA_pg^{-1}$ For all $i=1,2,...,m$. Therefore $tr(A'_i)=\sum_{p=1}^n h_{pi}tr(A_p)$, that is $\overline{tr}(A')=\overline{tr}(A)h.$ Similarly, 
		$A'_{i_1}A'_{i_2}...A'_{i_{k}}=\sum_{p_1,p_2,...,p_{k}=1}^n h_{p_1i_1}h_{p_2i_2}...h_{p_ki_k}gA_{p_1}A_{p_2}...A_{j_{k}}g^{-1}$ implies  $$tr(B_{i_1}B_{i_2}...B_{i_{k}})=\sum_{p_1,p_2,...,p_{k}=1}^nh_{p_1i_1}h_{p_2i_2}...h_{p_ki_k}tr(A_{p_1}A_{p_2}...A_{p_{k}})$$ for all $1\leq i_j\leq m$. By introduction of notation  $\overline{tr}((A')^{[k]})=\overline{tr}(A^{[k]})h^{\otimes k}$, where  $A^{[1]}=A=(A_1|A_2|...|A_n),\ \ A^{ [k]}=(A_1A^{ [k-1]}|A_2A^{[k-1]}|...|A_nA^{[k-1]})$ one can represent all those equalities
		in the following $\overline{tr}((A')^{[k]})=\overline{tr}(A^{[k]})h^{\otimes k}$  matrix form.\end{proof}

	Let $G$ be a group, $V$ be a set $\tau: (G,V)\rightarrow V$, be a representation of $G$. Note that elements $v,w\in V$ are said to be $G$-equivalent, with respect to the representation $\tau$, if $w=\tau(g,v)$ for some $g\in G$. 
	
	{\bf Assumption.} There exists a non-empty $G$-invariant subset $V_0$ of $V$ and a map $P: V_0\rightarrow G$ such that
	$P(\tau( g,\mathbf{v}))=P(\mathbf{v})g^{-1}$ whenever  $\mathbf{v}\in V_0$ and $g\in G$.
	
	\begin{lemma} Elements $\mathbf{u},\mathbf{v}\in V_0$ are $G$-equivalent, that is $\mathbf{u}=\tau( g,\mathbf{v})$ for some $g\in G$, if and only if $\tau( P(\mathbf{u}),\mathbf{u})=\tau( P(\mathbf{v}),\mathbf{v})$.\end{lemma} 
	
	\begin{proof}  If $\mathbf{u}=\tau( g,\mathbf{v})$ then $\tau( P(\mathbf{u}),\mathbf{u})=\tau( P(\tau( g,\mathbf{v})),\tau( g,\mathbf{v}))=$ \[\tau( P(\mathbf{v}) g^{-1},\tau( g,\mathbf{v}))=\tau( P(\mathbf{v}),\tau( g^{-1},\tau( g,\mathbf{v})))=\tau( P(\mathbf{v}),\mathbf{v}).\]
		Visa versa, if $\tau( P(\mathbf{u}),\mathbf{u})=\tau( P(\mathbf{v}),\mathbf{v})$ then \[ \tau( P(\mathbf{u})^{-1}P(\mathbf{v}),\mathbf{v})=\tau(( P(\mathbf{u}))^{-1},\tau( P(\mathbf{v}),\mathbf{v}))= \tau(( P(\mathbf{u}))^{-1},\tau( P(\mathbf{u}),\mathbf{u}))=\mathbf{u}\] that is $\mathbf{u}=\tau( g,\mathbf{v})$, where  $g=P(\mathbf{u})^{-1}P(\mathbf{v})\in G$.\end{proof}

	\section{The main result}
	Let $\mathbb{A}$ be an $n$-dimensional algebra over $\mathbb{F}$ and $\mathbf{e}=\{\mathrm{e}_1,\mathrm{e}_2,...,\mathrm{e}_n\}$ be a fixed basis for $\mathbb{A}$ as a vector space over $\mathbb{F}$. One can attach to the algebra $\mathbb{A}$  a $n \times n^2$-size matrix (called the matrix of structure constants, shortly MSC) $$A=\left(\begin{array}{ccccccccccccc}a_{11}^1&a_{12}^1&...&a_{1n}^1&a_{21}^1&a_{22}^1&...&a_{2n}^1&...&a_{n1}^1&a_{n2}^1&...&a_{nn}^1\\ a_{11}^2&a_{12}^2&...&a_{1n}^2&a_{21}^2&a_{22}^2&...&a_{2n}^2&...&a_{n1}^2&a_{n2}^2&...&a_{nn}^2 \\
		...&...&...&...&...&...&...&...&...&...&...&...&...\\ a_{11}^n&a_{12}^n&...&a_{1n}^n&a_{21}^n&a_{22}^n&...&a_{2n}^n&...&a_{n1}^n&a_{n2}^n&...&a_{nn}^n\end{array}\right)=(A_1|A_2|...|A_n)$$ as follows
	$$\mathrm{e}_i \cdot \mathrm{e}_j=\sum\limits_{k=1}^n a_{ij}^k\mathrm{e}_k, \ \mbox{where}\ i,j=1,2,...,n,\ \mbox{ in other words}\ e_i\cdot \mathbf{e}=\mathbf{e}A_i, i=1,2,...,n,$$ where $A_i=(a_{ij}^k)_{k,j=1,2,...,n}$. 
	
	In reality there is one more (opposite) way of attaching such a matrix to $\mathbb{A}$, namely, let $\mathbf{e}\cdot e_k=\mathbf{e}A^o_{k}, k=1,2,...,n,$ and $A^o=(A^o_1|A^o_2|...|A^o_n)$, where, in above notations, $A^o_{k}=(a_{jk}^i)_{i,j=1,2,...,n}$. so
	$$e_i\cdot \mathbf{e}=\mathbf{e}A_i,\ \ \mathbf{e}\cdot e_i=\mathbf{e}A^o_{i}\ \ i=1,2,...,n.$$ 
	The product operation of $\mathbb{A}$ with respect to this basis is written as follows
	\begin{equation} \label{Product}
		\mathrm{x}\cdot \mathrm{y}=\mathbf{e} A(x\otimes y),\ \mbox{as well as}\ \mathrm{x}\cdot \mathrm{y}=\mathbf{e} A^o(y\otimes x) 
	\end{equation}  
	for any $\mathbf{x}=\mathbf{e}x,\mathbf{y}=\mathbf{e}y,$
	where $x=(x_1, x_2,...,x_n)\in \mathbb{F}^n,$  $y=(y_1, y_2,...,y_n)\in \mathbb{F}^n$ are column coordinate vectors of $\mathbf{x}$ and $\mathbf{y},$ respectively, $x\otimes y$ is the tensor(Kronecker)
	product of the vectors $x$ and $y$. The behaviour of $A$ and $A^o$ with respect to a basis change  looks like $$\tau(g,A)=gA(g^{-1})^{\otimes 2}, \ \ \tau(g,A^o)=gA^o(g^{-1})^{\otimes 2}$$, where $g\in GL(n, \mathbb{F})$. Therefore further we look at $\mathcal{A}=\mathcal{A}(n)=Mat(n\times n^2, \mathbb{F})$-the variety of $n$-dimensional algebras only with respect to this action of $G=GL(n, \mathbb{F})$. 
	
	Further it is assumed that a basis $\mathbf{e}$ is fixed and we do not make difference between an algebra $\mathbb{A}$ and its MSC $A$, between a linear map and its matrix with respect to this basis.
	
	One can consider the following left and right invariant, symmetric Killing forms over $\mathbb{A}$: $$\mathbf{B}(\mathbf{x},\mathbf{y})=tr(ad^l_{\mathbf{x}}\circ ad^l_{\mathbf{y}})=tr(ad^l_{\mathbf{y}}\circ ad^l_{\mathbf{x}}), \ \mbox{where}\ ad^l_{\mathbf{x}}: \mathbb{A}\rightarrow \mathbb{A},\ ad^l_{\mathbf{x}}(\mathbf{z})=\mathbf{x}\mathbf{z},$$
	$$\mathbf{B}^o(\mathbf{x},\mathbf{y})=tr(ad^r_{\mathbf{x}}\circ ad^r_{\mathbf{y}})=tr(ad^r_{\mathbf{y}}\circ ad^r_{\mathbf{x}}),
	\ \mbox{where}\ ad^r_{\mathbf{x}}: \mathbb{A}\rightarrow \mathbb{A},\ ad^r_{\mathbf{x}}(\mathbf{z})=\mathbf{z}\mathbf{x}.$$
	
	Note that due to the equalities $$\mathbf{B}(e_i,e_j)=tr(ad^l_{e_i}\circ ad^l_{e_j})=tr(A_iA_j),\ \ \mathbf{B}^o(e_i,e_j)=tr(ad^r_{e_i}\circ ad^r_{e_j})=tr(A^o_{i}A^o_{j})$$ one has $$\mathbf{B}(\mathbf{x},\mathbf{y})=x^tB(A)y,\ \ \mathbf{B}^o(\mathbf{x},\mathbf{y})=x^tB^o(A)y,$$ where $B(A)=(tr(A_iA_j))_{i,j}, B^o(A)=(tr(A^o_iA^o_j))_{i,j}.$ Therefore, due to the Propositions 2.1,2.2, the following equalities 
	
	$$B(A')=(g^{-1})^tB(A)g^{-1}, B^o(A')=(g^{-1})^tB^o(A)g^{-1}$$ whenever $A'=gA(g^{-1})^{\otimes 2},\ g\in GL(n,\mathbb{F})$, are true. 
	
	Consider $\mathcal{A}_0=\mathcal{A}_0(n)=\{A\in\mathcal{A} :\  \det(B(A))\neq 0
	\}$-an invariant subset of $\mathcal{A}$.
		
	
	If $A\in\mathcal{A}_0$ and $A'=gA(g^{-1})^{\otimes 2},\ g\in GL(n,\mathbb{F})$ then $B(A')^{-1}=gB(A)^{-1}g^t, B(A')^{-1}B^o(A')=gB(A)^{-1}B^o(A)g^{-1}$ and $(B(A')^{-1}B^o(A'))^k=g(B(A)^{-1}B^o(A))^kg^{-1}$. Therefore one has the following important equalities  
	$$(B(A')^{-1}B^o(A'))^kA'=g(B(A)^{-1}B^o(A))^kA(g^{-1})^{\otimes 2},$$ $$\overline{tr}((B(A')^{-1}B^o(A'))^kA')=\overline{tr}((B(A)^{-1}B^o(A))^kA)g^{-1},$$ where $k$ is any non-negative integer.	
	So for the matrix $$P(A)=(\overline{tr}((B^{-1}(A)B^o(A))^0A)/\overline{tr}((B^{-1}(A)B^o(A))A)/.../\overline{tr}((B^{-1}(A)B^o(A))^{n-1}A))$$ the equality $$P(A')=P(A)g^{-1}$$ is true, where $(B^{-1}(A)B^o(A))^0=I$.
	
	\begin{example} Let $n=1, A=(a)$. In this case $A=A^o$, $B(A)=B^o(A)=(a^2)$, \\ 
		$\mathcal{A}_0=\{A=(a):\ 0\neq a\in\mathbb{F}\}, P(A)=(a)$..\end{example}
	\begin{example} Let $n=2, A=(A_1|A_2)=\begin{pmatrix}1&0&0&0\\
			0&1&0&1\end{pmatrix}$. In this case\\ $A^o=(A^o_1|A^o_2)=\begin{pmatrix}1&0&0&0\\
			0&0&1&1\end{pmatrix}$, $B(A)=\begin{pmatrix}2&1\\
			1&1\end{pmatrix}, B^o(A)=\begin{pmatrix}1&0\\
			0&1\end{pmatrix}$,\\ $C=B^{-1}(A)B^o(A)=\begin{pmatrix}1&-1\\
			-1&2\end{pmatrix}$, $\overline{tr}(CA)=\begin{pmatrix}3&2\end{pmatrix}$ and $P(A)=\begin{pmatrix}2&1\\
			3&2\end{pmatrix}, A\in   
		\mathcal{A}_{0}$.\end{example}
	\begin{example} Let $n=3, A=(A_1|A_2|A_3)=\begin{pmatrix}1&0&0&0&0&0&0&0&0\\
			0&1&0&0&1&0&0&0&0\\
			0&0&1&0&0&1&0&0&1\end{pmatrix}$. In this case $A^o=(A^o_1|A^o_2|A^o_3)=\begin{pmatrix}1&0&0&0&0&0&0&0&0\\
			0&0&0&1&1&0&0&0&0\\
			0&0&0&0&0&0&1&1&1\end{pmatrix}$, $B(A)=\begin{pmatrix}3&2&1\\ 2&2&1\\
			1&1&1\end{pmatrix}$, $B^o(A)=\begin{pmatrix}1&0&0\\ 0&1&0\\
			0&0&1\end{pmatrix}$, $C=B^{-1}(A)B^o(A)=\begin{pmatrix}	1&-1&0\\
			-1&2&-1\\
			0&-1&2&\end{pmatrix}$, $\overline{tr}(CA)=\begin{pmatrix}5&4&2\end{pmatrix}$, $\overline{tr}(C^2A)=\begin{pmatrix}13&11&5\end{pmatrix}$, and  $P(A)=\begin{pmatrix}3&2&1\\ 5&4&2\\
			13&11&5\end{pmatrix}, A\in   
		\mathcal{A}_{0}$. \end{example}
	\begin{lemma} The set $\mathcal{A}_{00}=\{A\in\mathcal{A}_0 :\ \det(P(A))\neq 0\}$ is not empty .
		
	\end{lemma}
	
	\begin{proof} Consider $A=(A_1|A_2|...|A_n)$, where $A_1=I_n$, $A_2=\begin{pmatrix}0&0\\
			0&I_{n-1}\end{pmatrix}$,$A_3=\begin{pmatrix}0&0\\
			0&I_{n-2}\end{pmatrix}$,..., $A_n=\begin{pmatrix}0&0\\
			0&1\end{pmatrix}$. In this case $B^o(A)=I_n, B(A)=\begin{pmatrix}n&n-1&n-2&n-3&...&2&1\\
			n-1&n-1&n-2&n-3&...&2&1\\
			n-2&n-2&n-2&n-3&...&2&1\\
			n-3&n-3&n-3&n-3&...&2&1\\
			.&.&.&.&...&.&.\\
			2&2&2&2&...&2&1\\
			1&1&1&1&...&1&1\end{pmatrix}$ and \\
		$C=B^{-1}(A)=\begin{pmatrix}	1&-1&0&0&...&0&0\\
			-1&2&-1&0&...&0&0\\
			0&-1&2&-1&...&0&0\\
			0&0&-1&2&...&0&0\\
			.&.&.&.&...&.&.\\
			.&.&.&.&...&2&-1\\
			0&0&0&0&...&-1&2\end{pmatrix}$, that is $A\in\mathcal{A}_{0}$. Moreover\\ $\overline{tr}(C^{k}A)=(\sum_{i=1}^n(C^{k})_{ii}, \sum_{i=2}^n(C^{k})_{ii},...,(C^{k})_{nn})$ therefore $\det(P(A))$ and determinant of the matrix $V(C)$ consisting of rows $((C^{k})_{11}, (C^{k})_{22},...,(C^{k})_{nn})$, $k=0,1,...,n-1$, where $(C^{0})_{11}=(C^{0})_{22}=...=(C^{0})_{nn}=1$, are equal. But according to Lemma 2.1 one knows that $\det(V(C))=\pm 1$, depending on $n$, so $A\in\mathcal{A}_{00}$.  \end{proof}
	
	The following main result is an direct application of Lemma 2.4 to our case.
	
	\begin{theorem} Two algebras $A,A'\in\mathcal{A}_{00}$ are isomorphic if and only if the equality $$P(A)A(P(A)^{-1})^{\otimes 2}=P(A')A'(P(A')^{-1})^{\otimes 2}$$ is valid. 	
	\end{theorem}
	
	\begin{corollary} In algebraic closed field $F$ case $\mathcal{A}_{00}$ is an open dense subset of $\mathcal{A}$ \end{corollary}

	\begin{corollary} For every $A\in \mathcal{A}_{00}$ there exists only one  $A'\in \mathcal{A}_{00}$ which is equivalent to $A$ and $P(A')=I$. So $A\in \mathcal{A}_{00}$ for which $P(A)=I$ can be taken as canonical representatives of orbits from $\mathcal{A}_{00}$. \end{corollary}
	\begin{proof} Indeed, if $A\in \mathcal{A}_{00}$ then  $A\simeq A'=P(A)A(P(A)^{-1})^{\otimes 2}$ and $P(A')=P(P(A)A(P(A)^{-1})^{\otimes 2})=P(A)P(A)^{-1}=I$. \end{proof}
	
	The direct sum of two algebras $\mathbb{D}=\mathbb{A}\oplus\mathbb{A'}$ in terms of theirs MSC is nothing but $D=A\oplus A'$, where $A=(A_1|A_2|...|A_n), A'=(A'_1|A'_2|...|A'_{n'})$, $D_i=\begin{pmatrix}A_i&0\\
		0&0\end{pmatrix}$ if $1\leq i\leq n$  and $D_i=\begin{pmatrix}0&0\\
		0&A'_{i-n}\end{pmatrix}$ if $n< i\leq n+n'$.
	
	\begin{theorem} If  $\mathbb{A}\in\mathcal{A}(n)$, $\mathbb{A}'\in\mathcal{A}(n')$ then $\mathbb{D}=\mathbb{A}\oplus\mathbb{A'}\in\mathcal{A}(n+n')$ if and only if  $A\in\mathcal{A}_{00}(n)$, $A'\in\mathcal{A}_{00}(n')$.
	\end{theorem}
	\begin{proof} Indeed, $tr(D_iD_j)=tr(A_iA_j)$ whenever $i,j\leq n$, $tr(D_iD_j)=tr(A'_iA'_j)$ whenever $i,j> n$ and otherwise $tr(D_iD_j)=0$. Therefore $B(D)=\begin{pmatrix}B(A)&0\\
			0&B(A')\end{pmatrix}$ is invertible if and only if both $B(A),B(A')$ are invertible, $P(D)=\begin{pmatrix}P(A)&0\\
			0&P(A')\end{pmatrix}\in \mathcal{A}_{00}(n+n')$ if and only if $P(A)\in \mathcal{A}_{00}(n)$, $P(A')\in \mathcal{A}_{00}(n')$.
	\end{proof}

	\begin{remark} Not every element of $\mathcal{A}_{00}(2)$ is a semi-simple algebra.
		
		Let $A=(A_1|A_2)=\begin{pmatrix}1&0&1&-1\\
			2&-1&0&0\end{pmatrix}$. This algebra has only the following non-trivial two-sided ideal $H=\mathbb{F}(e_1+e_2)$ and $A^o=(A^o_1|A^o_2)=\begin{pmatrix}1&1&0&-1\\
			2&0&-1&0\end{pmatrix}$, $B(A)=\begin{pmatrix}2&-1\\
			-1&1\end{pmatrix}, B^o(A)=\begin{pmatrix}5&-3\\
			-3&2\end{pmatrix}$, $C=B^{-1}(A)B^o(A)=\begin{pmatrix}2&-1\\
			-1&1\end{pmatrix}$, $\overline{tr}(CA)=\begin{pmatrix}-1&3\end{pmatrix}$ and $P(A)=\begin{pmatrix}0&1\\
			-1&3\end{pmatrix}$, $A\in   
		\mathcal{A}_{00}(2)$.
	\end{remark}

	\begin{remark} Not every simple two-dimensional  algebra is in $\mathcal{A}_{00}(2)$.
		
		Let $A=(A_1|A_2)=\begin{pmatrix}0&1&1&0\\
			1&0&0&-1\end{pmatrix}$-a Jordan algebra. It is a simple algebra, but it is not in $\mathcal{A}_{00}(2)$, $P(A)$ is singular as far as $\overline{tr}(A)=0$.	\end{remark}
	
	\begin{theorem} If  $\mathbb{A}\in\mathcal{A}(n)$, $\mathbb{A}'\in\mathcal{A}(n')$ then $\mathbb{D}=\mathbb{A}\otimes\mathbb{A'}\in\mathcal{A}(nn')$ if and only if  $A\in\mathcal{A}_{00}(n)$, $A'\in\mathcal{A}_{00}(n')$.
	\end{theorem}
	\begin{proof} Indeed, if $e=(e_1,e_2,...,e_n)$, $e'=(e'_1,e'_2,...,e'_{n'})$ are corresponding bases then $e\otimes e'$ is a basis for $\mathbb{D}$ and its MSC is \[D=A\otimes A'=(A_1\otimes A'_1|...|A_1\otimes A'_{n'}|A_2\otimes A'_1|...|A_2\otimes A'_{n'}|...|A_n\otimes A'_1|...|A_n\otimes A'_{n'}),\] $\overline{tr}(D)=\overline{tr}(A)\otimes\overline{tr}(A')$,  $tr((A_i\otimes A'_p)(A_j\otimes A'_q))=tr(A_iA_j\otimes A'_pA'_q)=tr(A_iA_j)tr(A'_pA'_q)$. It implies that $B(D)=B(A)\otimes B(A')$ and $B(D)$ is non-singular if and only if $B(A), B(A')$ are nonsingular. Similarly, $B(D^o)=B(A^o)\otimes B(A'^o)$, $(B(A)\otimes B(A'))^{-1}(B(A^o)\otimes B(A'^o))=B(A)^{-1}B(A^o)\otimes B(A')^{-1}B(A'^o)$, $P(D)=P(A)\otimes P(A')$ and $P(D)$ is invertible if and only $P(A), P(A')$ are invertible.  
	\end{proof}
	
	Note that if $a\in \mathbb{F}$ is non-zero, $\mathbb{A}\in \mathcal{A}_{00}(n)$ then $a\mathbb{A}\in \mathcal{A}_{00}(n)$, where $a\mathbb{A}$ stands for the algebra with MSC $(aA_1|aA_2|...|aA_n)$,  as far as $P(aA)=aP(A)$. Therefore consideration of  $\cup_{n=0}^{\infty}\mathcal{A}_{00}(n)$ with respect to $\oplus, \otimes$ and multiplication by a scalar operations is a natural thing, where $\mathcal{A}_{00}(0)$ stands for the zero element. Isn't it an interesting object to explore? 
	
	One can use the classification theorem from \cite{BU2} to describe elements of $\mathcal{A}_{00}(2)$ up to isomorphism. Here we present the corresponding result in $Char(\mathbb{F})\neq 2,3$ case and follow the notations of that paper.  
	\begin{theorem} Every element of $\mathcal{A}_{00}(2)$, in $Char(\mathbb{F})\neq 2,3$ case, is isomorphic to only one of the following algebras from $\mathcal{A}_{00}(2)$ given by theirs following MSC's
		\begin{itemize}
			\item $A_{1}(\mathbf{c})=\left(
			\begin{array}{cccc}
				\alpha_1 & \alpha_2 &1+\alpha_2 & \alpha_4 \\
				\beta_1 & -\alpha_1 & 1-\alpha_1 & -\alpha_2
			\end{array}\right),\ \mbox{where}\ \mathbf{c}=(\alpha_1, \alpha_2, \alpha_4, \beta_1) \in \mathbb{F}^4, 4\alpha_2^3\beta_1+3\alpha_1^2\alpha_2^2-2\alpha_1\alpha_2^2+4\alpha_2^2\beta_1+\alpha_2^2+2\alpha_1^2\alpha_2-2\alpha_1\alpha_2 +4\alpha_1^2\alpha_4-4\alpha_1^3\alpha_4-2\alpha_1\alpha_4-6\alpha_1\alpha_2\alpha_4\beta_1-2\alpha_1\alpha_4\beta_1-\alpha_1^2-\alpha_4^2\beta_1^2+2\alpha_2\alpha_4\beta_1\neq 0,$\\ $\alpha_1-2\alpha_1^2+2\alpha_1^3
			+2\alpha_1\alpha_2-4\alpha_1^2\alpha_2+4 \alpha_1^3\alpha_2+ 			
			2\alpha_1\alpha_2^2-4\alpha_1^2\alpha_2^2+2\alpha_1\alpha_4-6\alpha_1^2\alpha_4			
			+8\alpha_1^3\alpha_4 +2\alpha_1\beta_1+6\alpha_1						\alpha_2\beta_1+8\alpha_1\alpha_2^2\beta_1+			
			4\alpha_1\alpha_4\beta_1-4\alpha_1^2\alpha_4\beta_1			
			+ 4\alpha_1\alpha_2\alpha_4\beta_1\neq 0,$

			\item $A_{2}(\mathbf{c})=\left(
			\begin{array}{cccc}
				\alpha_1 & 0 & 0 & \alpha_4 \\
				1& \beta_2 & 1-\alpha_1&0
			\end{array}\right),\ \mbox{where}\ \mathbf{c}=(\alpha_1,\alpha_4, \beta_2)\in \mathbb{F}^3,\\ \alpha_4(2\alpha_1^3 - 2\alpha_1^2 + 2\alpha_1\beta_2^2 + \alpha_4 - 2\beta_2^2)(\alpha_1+\beta_2-1)\neq 0,$
			\item $A_{4}(\mathbf{c})=\left(
			\begin{array}{cccc}
				0 & 1 & 1 & 0 \\
				\beta _1& \beta_2 & 1&-1
			\end{array}\right),\ \mbox{where}\ \mathbf{c}=(\beta_1, \beta_2)\in \mathbb{F}^2,$\\ $(\beta_2^2+2\beta_1+2\beta_2-1)(-2\beta_2^2+4\beta_2+5)\neq 0,$
			\item $A_{6}(\mathbf{c})=\left(
			\begin{array}{cccc}
				\alpha_1 & 0 & 0 & \alpha_4 \\
				1& 1-\alpha_1 & -\alpha_1&0
			\end{array}\right),\ \mbox{where}\ \mathbf{c}=(\alpha_1,\alpha_4)\in \mathbb{F}^2,\\ \alpha_4(2(2\alpha_1^2-2\alpha_1+1)\alpha_1+\alpha_4)\neq 0,$ 
			\item $A_{8}(\mathbf{c})=\left(
			\begin{array}{cccc}
				0 & 1 & 1 & 0 \\
				\beta _1& 1 & 0&-1
			\end{array}\right),\ \mbox{where}\ \mathbf{c}=\beta_1\in\mathbb{F},$ $4\beta_1+1\neq 0.$	
		\end{itemize}
	\end{theorem}
	\begin{proof} Proof consists of calculations of $B(A), P(A)$, finding out for which $A$ from the classification Theorem 1 \cite{BU2} the inequalities $\det(B(A))\neq 0, \det(P(A))\neq 0 $ are valid. \end{proof}

	\begin{definition} A linear map $\mathbf{g}: \mathbb{A}\rightarrow\mathbb{A}$ $(\mathbf{D}: \mathbb{A}\rightarrow\mathbb{A}$) is said to be an automorphism (respectively,  a derivation) if $\mathbf{g}$ is invertible and for any $x,y\in \mathbb{A}$ the equality $\mathbf{g}(xy)=\mathbf{g}(x)\mathbf{g}(y)$  (respectively, $\mathbf{D}(xy)=\mathbf{D}(x)y+x\mathbf{D}(y)$) holds true.\end{definition}
	
	Note that in terms of its matrix $g$ (respectively, $D$) it is nothing but
	$$A=gA(g^{-1})^{\otimes 2}, (\mbox{respectively},\ DA=A(D\otimes I+I\otimes D)).$$ 
	
	We use notation $Aut(\mathbb{A})$ $(Der(\mathbb{A}))$ for the set of all automorphisms(respectively, derivations) of the algebra $\mathbb{A}$.

	\begin{corollary} If $A\in \mathcal{A}_{00}$ then $Aut(A)=\{I\}$. \end{corollary}
	\begin{proof} Indeed, if $A=gA(g^{-1})^{\otimes 2}$ then  $P(A)=P(A)g^{-1}$ that is $g=I$.\end{proof}
	
	\begin{corollary} If $Char(\mathbb{F})\neq 2,3$ and $A\in \mathcal{A}_{00}(2)$ then $Der(A)=\{0\}$. \end{corollary}
	
	\begin{example} If $Char(\mathbb{F})\neq 2$  then for the algebra $A=\left(
		\begin{array}{cccc}
			1 & 0 &1 & -\frac{1}{2} \\
			0& -1 & 0 & 0
		\end{array}\right)$ one has $Aut(A)=\{I\}$, $Der(A)=\{0\}$, but $A$ is not in  $\mathcal{A}_{00}(2)$.
		
		In $Char(\mathbb{F})= 2$  case  the algebra 
		$A=\left(
		\begin{array}{cccc}
			1 & \alpha_2 &1+\alpha_2 & 1 \\
			1 & -1 & 0 & -\alpha_2
		\end{array}\right)$ has similar property: $Aut(A)=\{I\}$, $Der(A)=\{0\}$, but $A$ is not in  $\mathcal{A}_{00}(2)$..	\end{example}

	\begin{corollary} If $A\in \mathcal{A}$ has property that $\overline{tr}(A)\neq 0$ or $\overline{tr}(A^o)\neq 0$ then none of the elements of $Der(A)$ is invertible as a linear map. In particular, $Der(A)$ has no invertible element whenever $A\in \mathcal{A}_{00}$. \end{corollary}
	\begin{proof} 
		If $\mathbf{D}: \mathbb{A}\rightarrow\mathbb{A}$ is a derivation  with matrix $(d_{ij})$ then equalities
		$DA=A(D\otimes I+I\otimes D)$, $DA^o=A^o(D\otimes I+I\otimes D)$ are valid. They imply equalities  $DA_i=A_iD+\sum_{k=1}^nA_kd_{ki}$, $DA^o_i=A^o_iD+\sum_{k=1}^nA^o_kd_{ki}$ 	 
		, where $i=1,2,..,n$. Taking traces results in  $tr(DA_i)=tr(A_iD)+\sum_{k=1}^ntr(A_k)d_{ki}$, $tr(DA^o_i)=tr(A^o_iD)+\sum_{k=1}^ntr(A^o_k)d_{ki}$, $i=1,2,..,n$.
		Due to $tr(DA_i)=tr(A_iD)$ one can represent these equalities in the following matrix forms $\overline{tr}(A)D=0, \overline{tr}(A^o)D=0$. Therefore if $D$ is invertible then $\overline{tr}(A)=\overline{tr}(A^o)=0$.\end{proof}
	
	\section{Some other open subsets with similar properties}
	
	One disadvantage of the introduced $mathcal{A}_{00}$ is that it contains neither commutative, no anti-commutative algebras in $n>1$ case as far as in these cases $B(A)=\pm B^o(A)$ and therefore $P(A)$ is singular. In this sense the following approach may be better.
	
	Due to $A'=gA(g^{-1})^{\otimes 2}, B(A')=(g^{-1})^tB(A)g^{-1}$ it is easy to check that the equality 
	$$A'(B(A')^{-1})^{\otimes 2}A'^tB(A')^t=gA(B(A)^{-1})^{\otimes 2}A^tB(A)^tg^{-1}$$ holds true, provided that $A\in \mathcal{A}_{0}$.
	It implies that the matrix $$Q(A)=(\overline{tr}(A)/\overline{tr}(C(A)A)/\overline{tr}(C(A)^2A)/.../\overline{tr}(C(A)^{n-1}A))$$, where $C(A)=A(B(A)^{-1})^{\otimes 2}A^tB(A)^t$, has property $$Q(gA(g^{-1})^{\otimes 2})=Q(A)g^{-1}.$$ Therefore the following result is true, where $\mathcal{A}_{01}=\{A\in\mathcal{A}_{0}:\ \det(Q(A))\neq 0 \}.$
	
	\begin{theorem} Two algebras $A,A'\in\mathcal{A}_{01}$ are isomorphic if and only if the equality $$Q(A)A(Q(A)^{-1})^{\otimes 2}=Q(A')A'(Q(A')^{-1})^{\otimes 2}$$ is valid. 	
	\end{theorem}
	
	The following result result shows that, in general, $\mathcal{A}_{00}$ and $\mathcal{A}_{01}$ are different sets.
	
	\begin{example} Let $n=2, A=(A_1|A_2)=\begin{pmatrix}1&0&0&1\\
			b&0&0&0\end{pmatrix}$- a commutative algebra. In this case\\ $B(A)=\begin{pmatrix}1&b\\
			b&0\end{pmatrix}, B(A)^{-1}=-\frac{1}{b^2}\begin{pmatrix}0&-b\\
			-b&1\end{pmatrix}$,\  $ (B(A)^{-1})^{\otimes 2}=\frac{1}{b^4}\begin{pmatrix}0&0&0&b^2\\
			0&0&b^2&-b\\ 	0&b^2&0&-b\\ 	b^2&-b&-b&1\end{pmatrix}$,  $A(B(A)^{-1})^{\otimes 2}=\frac{1}{b^4}\begin{pmatrix}b^2&-b&-b&1+b^2\\
			0&0&0&b^3\\\end{pmatrix}$, $A(B(A)^{-1})^{\otimes 2}A^t=\frac{1}{b^4}\begin{pmatrix}1+2b^2&b^3\\
			b^3&0\\\end{pmatrix}$, $C=A(B(A)^{-1})^{\otimes 2}A^tB(A)^t=\frac{1}{b^4}\begin{pmatrix}1+2b^2+b^4&b(1+b^2)\\
			b^3&b^4\\\end{pmatrix}$, $Q(A)=\begin{pmatrix}1&0\\ \frac{1+3b^2+2b^4}{b^4}&\frac{1}{b}\end{pmatrix}$, so $A\in\mathcal{A}_{01}(2)$, whenever $b\neq 0$. Note that this algebra is not in $\mathcal{A}_{00}(2)$. In general, $\mathcal{A}_{00}(2)$ is not a subset of $\mathcal{A}_{00}(2)$, which one can see considering $A=A_8(\beta_1)$.\end{example}
		
		The next result shows that if $Char(\mathbb{F})=0$ or at least  
		$Char(\mathbb{F})\geq n$ then $\mathcal{A}_{01}$ is not empty.		
		\begin{lemma} If $Char(\mathbb{F})=0$ or  
			$Char(\mathbb{F})\geq n$ then $\mathcal{A}_{01}$ is not empty
		\end{lemma}
		\begin{proof} Let $A=(A_1|A_2|...|A_n)$, where for every $i=1,2,...,n$ $A_i$'s $i^{th}$ rows first $i$ elements are $1$, all other entries of $A_i$ are zero In this case $B(A)=I$, $C=AA^t=Diag(1,2,...,n)$ and $Q(A)$ is the Vandermonde matrix $V(\lambda_1,\lambda_2,...,\lambda_n)$ with $\lambda_1=1, \lambda_2=2,...,\lambda_n=n$. Therefore if $Char(\mathbb{F})=0$ or $Char(\mathbb{F})\geq n$ then $Q(A)$ is nonsingular.
		\end{proof}
			
		As we know that the direct sum of two algebras $\mathbb{D}=\mathbb{A}\oplus\mathbb{A'}$ in terms of theirs MSC is nothing but $D=A\oplus A'$, where $A=(A_1|A_2|...|A_n), A'=(A'_1|A'_2|...|A'_{n'})$, $D_i=\begin{pmatrix}A_i&0\\
			0&0\end{pmatrix}$ if $1\leq i\leq n$  and $D_i=\begin{pmatrix}0&0\\
			0&A'_{i-n}\end{pmatrix}$ if $n< i\leq n+n'$.
		
		\begin{theorem} If  $\mathbb{A}\in\mathcal{A}(n)$, $\mathbb{A}'\in\mathcal{A}(n')$ then $\mathbb{D}=\mathbb{A}\oplus\mathbb{A'}\in\mathcal{A}_{01}(n+n')$ if and only if  $A\in\mathcal{A}_{01}(n)$ and $A'\in\mathcal{A}_{01}(n')$.
		\end{theorem}
		\begin{proof} Indeed, $tr(D_iD_j)=tr(A_iA_j)$ whenever $i,j\leq n$, $tr(D_iD_j)=tr(A'_iA'_j)$ whenever $i,j> n$ and otherwise $tr(D_iD_j)=0$. Therefore $B(D)=\begin{pmatrix}B(A)&0\\
				0&B(A')\end{pmatrix}$ and it is invertible if and only if both $B(A),B(A')$ are invertible. Moreover $B(D)^{-1}=\begin{pmatrix}B(A)^{-1}&0\\
				0&B(A')^{-1}\end{pmatrix}$,\\ $(B(D)^{-1})^{\otimes 2}=\begin{pmatrix}B(A)^{-1}\otimes B(D)^{-1}&0\\
				0&B(A')^{-1}\otimes B(D)^{-1}\end{pmatrix}$,  $D(B(D)^{-1})^{\otimes 2}=\\ (\sum_{i=1}^n(B(A)^{-1})_{i1}\begin{pmatrix}A_iB(A)^{-1}&0\\
				0&0\end{pmatrix},...,\sum_{i=1}^n(B(A)^{-1})_{in}\begin{pmatrix}A_iB(A)^{-1}&0\\
				0&0\end{pmatrix},\\ \sum_{i=1}^{n'}(B(A')^{-1})_{i1}\begin{pmatrix}0&0\\
				0&A'_iB(A')^{-1}\end{pmatrix},...,\sum_{i=1}^{n'}(B(A')^{-1})_{in'}\begin{pmatrix}0&0\\
				0&A'_iB(A')^{-1}\end{pmatrix})=  \\ 
			(\begin{pmatrix}\sum_{i=1}^n(B(A)^{-1})_{i1}A_iB(A)^{-1}&0\\
				0&0\end{pmatrix},...,\begin{pmatrix}\sum_{i=1}^n(B(A)^{-1})_{in}A_iB(A)^{-1}&0\\
				0&0\end{pmatrix},\\ \begin{pmatrix}0&0\\
				0&\sum_{i=1}^{n'}(B(A')^{-1})_{i1}A'_iB(A')^{-1}\end{pmatrix},...,\begin{pmatrix}0&0\\
				0&\sum_{i=1}^{n'}(B(A')^{-1})_{in'}A'_iB(A')^{-1}\end{pmatrix})=\\
			(\begin{pmatrix}A(B(A)^{-1}\otimes B(A)^{-1}) &0\\
				0&A'(B(A')^{-1}\otimes B(A')^{-1})\end{pmatrix}$.\\ Therefore $C(D)=D(B(D)^{-1})^{\otimes 2}D^tB(D)^t=\\
			(\begin{pmatrix}A(B(A)^{-1}\otimes B(A)^{-1})A^tB(A)^t &0\\
				0&A'(B(A')^{-1}\otimes B(A')^{-1})A'^tB(A')^t\end{pmatrix}=
			(\begin{pmatrix}C(A)&0\\
				0&C(A')\end{pmatrix}$, \\
			$C(D)^kD=(\begin{pmatrix}C(A)^kA&0\\
				0&C(A')^kA'\end{pmatrix}$, $Q(D)=(\begin{pmatrix}Q(A)&0\\
				0&Q(A')\end{pmatrix}$, which implies that\\
			$Q(D)\in \mathcal{A}_{01}(n+n')$ if and only if $Q(A)\in \mathcal{A}_{01}(n)$ and $Q(A')\in \mathcal{A}_{01}(n')$.
		\end{proof}

	Further for the entries of $A\otimes B$ we use notation $(A\otimes B)_{(p,q)(i,j)}=A_{pi}B_{qj}$ where $p$($q$) is row number for $A$(respectively, $B$), $i$($j$) is column number for $A$(respectively, $B$)

	\begin{theorem} If  $\mathbb{A}\in\mathcal{A}(n)$, $\mathbb{A}'\in\mathcal{A}(n')$ then $\mathbb{D}=\mathbb{A}\otimes\mathbb{A'}\in\mathcal{A}_{01}(nn')$ if and only if  $A\in\mathcal{A}_{01}(n)$, $A'\in\mathcal{A}_{01}(n')$.
	\end{theorem}
	\begin{proof} Indeed, if $e=(e_1,e_2,...,e_n)$, $e'=(e'_1,e'_2,...,e'_{n'})$ are corresponding bases then $e\otimes e'$ is a basis for $\mathbb{D}$ and its MSC is \[D=A\otimes A'=(A_1\otimes A'_1|...|A_1\otimes A'_{n'}|A_2\otimes A'_1|...|A_2\otimes A'_{n'}|...|A_n\otimes A'_1|...|A_n\otimes A'_{n'}).\]
	Therefore	$D^t=$ \[A^t\otimes (A')^t=(A_1^t\otimes (A'_1)^t/.../A_1^t\otimes (A'_{n'})^t/A_2^t\otimes (A'_1)^t/.../A_2^t\otimes (A'_{n'})^t/.../A_n^t\otimes (A'_1)^t/.../A_n^t\otimes(A'_{n'})^t),\] $\overline{tr}(D)=\overline{tr}(A)\otimes\overline{tr}(A')$,  $tr((A_p\otimes A'_q)(A_i\otimes A'_j))=tr(A_pA_i\otimes A'_qA'_j)=tr(A_pA_i)tr(A'_qA'_j)$. It implies that $B(D)=B(A)\otimes B(A')$ and $B(D)$ is non-singular if and only if $B(A), B(A')$ are nonsingular. Moreover 
		$B(D)^{-1}=B(A)^{-1}\otimes B(A')^{-1}$, $D(B(D)^{-1})^{\otimes 2}=(\sum_{p=1,q=1}^{n,n'}(B(A)^{-1}\otimes B(A')^{-1})_{(p,q)(1,1)}(A_p\otimes A'_q)(B(A)^{-1}\otimes B(A')^{-1}),...,\sum_{p=1,q=1}^{n,n'}(B(A)^{-1}\otimes B(A')^{-1})_{(p,q)(n,n')}(A_p\otimes A'_q)(B(A)^{-1}\otimes B(A')^{-1}))=(\sum_{p=1,q=1}^{n,n'}(B(A)^{-1}\otimes B(A')^{-1})_{(p,q)(1,1)}(A_pB(A)^{-1}\otimes A'_qB(A')^{-1}),...,\\ \sum_{p=1,q=1}^{n,n'}(B(A)^{-1}\otimes B(A')^{-1})_{(p,q)(n,n')}(A_pB(A)^{-1}\otimes A'_qB(A')^{-1}))$ that is 
				$D(B(D)^{-1})^{\otimes 2}=\\ (\sum_{p=1,q=1}^{n,n'}(B(A)^{-1}\otimes B(A')^{-1})_{(p,q)(i,j)}(A_p\otimes A'_q)(B(A)^{-1}\otimes B(A')^{-1}))_{i=1,2,...,n, j=1,2,...,n'}=\\ (\sum_{p=1,q=1}^{n,n'}(B(A)^{-1}\otimes B(A')^{-1})_{(p,q)(i,j)}(A_pB(A)^{-1}\otimes A'_qB(A')^{-1}))_{i=1,2,...,n, j=1,2,...,n'}$-row block,\\
				$D(B(D)^{-1})^{\otimes 2}D^t=\\ \sum_{i=1,j=1}^{n,n'}\sum_{p=1,q=1}^{n,n'}(B(A)^{-1}\otimes B(A')^{-1})_{(p,q)(i,j)}(A_pB(A)^{-1}A_i^tB(A)^t\otimes A'_qB(A')^{-1}(A'_j)^tB(A')^t),$\\
				$D(B(D)^{-1})^{\otimes 2}D^tB(D)^t=\\	\sum_{i=1,j=1}^{n,n'}\sum_{p=1,q=1}^{n,n'}(B(A)^{-1}\otimes B(A')^{-1})_{(p,q)(i,j)}(A_pB(A)^{-1}A_i^tB(A)^t\otimes A'_qB(A')^{-1}(A'_j)^tB(A')^t)=\\ \sum_{i=1,j=1}^{n,n'}\sum_{p=1,q=1}^{n,n'}(B(A)^{-1})_{pi}B(A')^{-1})_{(q,j)}(A_pB(A)^{-1}A_i^tB(A)^t\otimes A'_qB(A')^{-1}(A'_j)^tB(A')^t)=\\ \sum_{p,i=1}^{n}B(A)^{-1}_{pi}A_pB(A)^{-1}A_i^tB(A)^t\otimes \sum_{q=1,j=1}^{n'}(B(A')^{-1})_{(q,j)}A'_qB(A')^{-1}(A'_j)^tB(A')^t=\\
		A(B(A)^{-1})^{\otimes 2}A^tB(A)^t\otimes A'(B(A')^{-1})^{\otimes 2}(A')^tB(A')^t.$ It means that
		$Q(D)=Q(A)\otimes Q(A')$ and $Q(D)$ is invertible if and only $Q(A)$ and $Q(A')$ are invertible.  
	\end{proof}


\begin{thebibliography}{9}
		\bibitem {P} 
		Vladimir L. Popov, 
		Algebras of General Type: 
		Rational Parametrization and Normal Forms, Proceedings of the Steklov Institute of Mathematics, 2016, Vol. 292, pp. 202–215, Pleiades Publishing, Ltd., 2016.
		
		\bibitem {BU1} Bekbaev U., On classification of finite dimensional algebras, ArXiv: 1504.01194v3, pp. 1--8
		
		\bibitem {BU2} Bekbaev U., Classification of two-dimensional algebras over any basic field, AIP Conference Proceedings 2880, 030001 (2023), pp. 1--9
		
		\bibitem{ABR1} Ahmed H., Bekbaev U., and Rakhimov I., The automorphism group and derivation algebras of two-dimensional algebras, 9 pages,  Journal of Generalized Lie Theory and Applications, 2018, 12-1, doi 10.4172/1736-4337.1000290.
		
		\bibitem{ABR2} Ahmed H., Bekbaev U., and Rakhimov I.,
		Identities on Two-Dimensional Algebras, 
		Lobachevskii Journal of Mathematics, 2020, Vol. 41, No. 9, pp. 1615–-1629
		
		\bibitem{N} Nicholson, William K. (1993), A short proof of the Wedderburn–Artin theorem. New Zealand J. Math. 22: pp. 83–-86.
		
		\bibitem{Bl} Block, Richard E.; Wilson, Robert Lee (1988), "Classification of the restricted simple Lie algebras", Journal of Algebra, 114 (1): pp. 115–-259
		
		
	\end{thebibliography}
\end{document}